\documentclass[pdflatex,sn-mathphys-num]{sn-jnl}
\usepackage{graphicx}%
\usepackage{multirow}%
\usepackage{amsmath,amssymb,amsfonts}%
\usepackage{amsthm}%
\usepackage{mathrsfs}%
\usepackage[title]{appendix}%
\usepackage{xcolor}%
\usepackage{textcomp}%
\usepackage{manyfoot}%
\usepackage{booktabs}%
\usepackage{algorithm}%
\usepackage{algorithmicx}%
\usepackage{algpseudocode}%
\usepackage{listings}%

\usepackage{breqn}

\theoremstyle{thmstyleone}

\theoremstyle{thmstyletwo}%

\theoremstyle{thmstylethree}%

\newtheorem{thm}{Theorem}

\newtheorem{cor}{Corollary}
\newtheorem{rem}{Remark}

\begin{document}
	
	\title[Article Title]{Feasibility problem for the radii of Steiner 4-chains}
	
	\author[]{A.Diakvnishvili}
	\affil[]{Faculty of Business,Technology and Education, Ilia State University, Tbilisi, Georgia}
	
	\abstract{Motivated by the solution of Descartes famous problem on the chains of tangent circles given by D.Mathews and O.Zymaris, we give an algorithmic solution to the feasibility problem for radii of Steiner $4$-chains studied by K.Kirajiev. More precisely, we obtain a complete characterization of the ordered quadruples of positive numbers which coincide with a quadruple of curvatures of the circles forming a Steiner $4$-chain. Our approach makes essential use of the numerical invariants of Steiner chains introduced in a recent paper of R.Schwartz and S.Tabachnikov. We compute all invariants of Steiner $4$-chains and establish the algebraic relations between those invariants which are used in the proof of the main result. Several related observations and illustrative examples are also given} 
	
	
	\maketitle
	\vspace*{0.25cm}\noindent{\small {\bf MSC 2010:} {52C35,  32S40.}}
	
	\section{Introduction}\label{sec1}
	
	A classical problem of Descartes’ formulated in his letters to Princess Palatine Elisabeth of Bohemia between 1643 and 1649 (cf. \cite{par}), in modern terms is concerned with the characterization of $n$-tuples of positive numbers which appear as the radii of a chain of externally tangent circles forming the so-called {\it $n$-flower} \cite{mazy}, which means that there exists a circle externally tangent to the circles considered. For brevity, we refer to this problem as Descartes Flower Problem (DFP). For $n=3$, the answer is trivial: any triple of positive numbers yields a solution to DFP. For arbitrary $n\geq 4$, highly nontrivial general necessary conditions were given in a seminal paper by  D.Mathews and O.Zymaris \cite{mazy}. Further related results in the context of circle packings can be found in preprint \cite{mazy2}. However, to the best of our knowledge a general criterion for the solvability of DFP, in other words a {\it feasibility criterion for radii of $n$-flowers}, in the general case is still unknown. \\
	
	A version of this problem in the case of Steiner $4$-chains was studied in a prize-winning paper of K.Kiradjiev of Oxford University \cite{kir}.  K.Kiradjiev described an approach based on the use of M\"obius transformations and gave conditions for a quadruple of positive numbers to coincide with a quadruple of curvatures of the circles forming a Steiner $4$-chain \cite{kir}. However the discussion in \cite{kir} is rather brief and the proof that the given conditions yield a genuine criterion is only outlined without giving the full detail. One of the aims of the present note is to give a detailed solution to the same problem based on the use of invariants of Steiner chains introduced in a recent paper of R.Schwartz and S.Tabachnikov \cite{scta}. Our approach is basically the same as in \cite{bisa, bidi, dikh} and relies on a few general results on the Steiner chains given in \cite{scta} and \cite{yiu}. We show that our results are consistent with the results of \cite{kir}, and our approach suggests further generalizations for bigger values of $n$. \\
	
	In the next section we present the necessary background on the invariants of Steiner chains studied in \cite{lamawi,scta}. In Section 3 we give the explicit formulas for the invariant bending moments of Steiner $4$-chains and express the Pedoe feasibility criterion \cite{ped} in terms of the first two bending moments (Theorem 2), which plays a crucial role in our approach. In the fourth section we give an analytic expression for the radii of Steiner circles (Theorem 4) which implies an important result from \cite{yiu} (Corollary 7) necessary for our purposes. In Section 5 we give an algorithmic solution for the aforementioned {\it feasibility problem for radii of Steiner $4$-chains} studied in \cite{kir} and \cite{bidi}. The last section contains several remarks on possible generalizations and research directions suggested by our results. \\
	
	\newpage
	\section{Background on Steiner chains}\label{sec2}
	
	As usual the term {\it Steiner $n$-chain} refers to a sequence of non-intersecting circles $\delta_i, i=1, \ldots, n,$ in Euclidean plane such that the circles with adjacent indices $mod\, n$ are externally tangent, each of $\delta_i$ is externally tangent to a fixed circle $\gamma$ and internally tangent to another fixed circle $\Gamma$ containing $\gamma$ in its interior disk (see Fig.1 for $n=4$). The nested circles $\gamma$ and $\Gamma$ are called the {\it parent circles} of the Steiner chain considered \cite{ped, scta}. If the inner circle $\gamma$ contains the center $\Omega$ of the outer circle $\Gamma$, this pair of parent circles is called {\it exact}. For such pairs of parent circles, certain results in the sequel take especially simple form. Steiner chains is a classical topic discussed in many papers, in particular, in the context of {\it Steiner porism} (see, e.g., \cite{ped,scta,bisa,bidi}) which is in the focus of our discussion.
	
	\begin{figure}[htbp]
		\centering
		\includegraphics[width=0.5\textwidth]{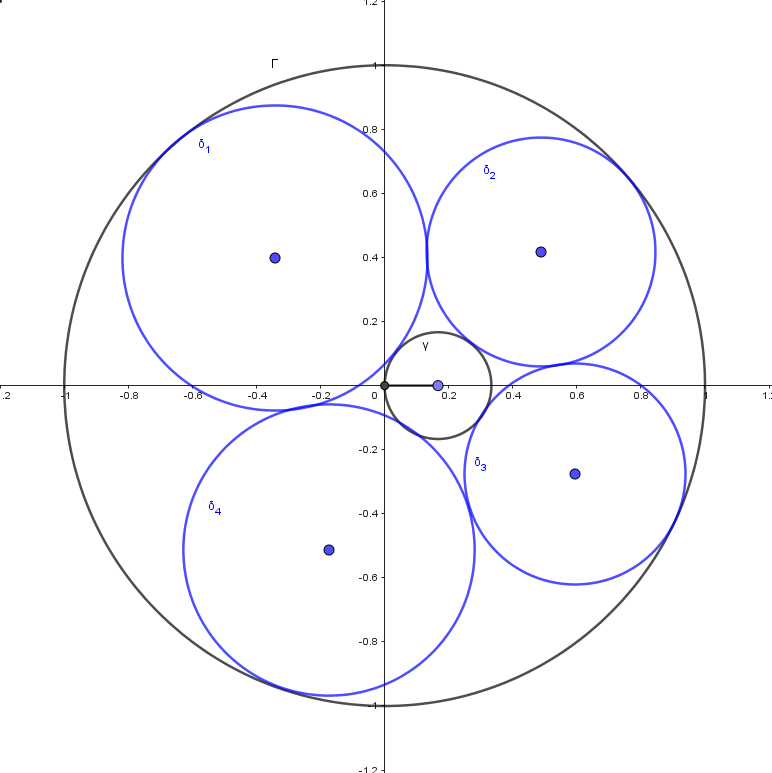}
		\caption{A Steiner $4$-chain}
		\label{graph1}
	\end{figure} 
	
	The aforementioned Steiner porism states that there exists a one-dimensional family of Steiner $n$-chains having the same parent circles \cite{ber,ped}. By a way of analogy with the Poncelet porism \cite{ber} a collection of Steiner $n$-chains with the fixed parent circles will be called {\it poristic Steiner $n$-chains}. This setting suggests, in particular, several extremal problems for poristic Steiner $n$-chains with the fixed parent circles. In particular, a natural extremal problem concerned with the sum of areas of poristic Steiner $4$-chains was studied in a prize-winning paper of K.Kiradjiev \cite{kir}. The results of K.Kiradjiev have been generalized in a joint paper of G.Bibileishvili and A.Diakvnishvili \cite{bidi}. The papers \cite{bidi,kir} as well as the already mentioned paper of D.Mathews and O.Zymaris \cite{mazy} served as an impetus for the research which led to the results given in the present paper. \\
	
	We add that triples of pairwise externally tangent circles, sometimes called "kissing circles", were historically the first source of Steiner chains. Such triples always have parent circles and give rise ot Steiner porism of order three. The famous Descartes' circle theorem enables one to compute the curvatures of the parent circles in terms of the curvatures of the three given "kissing circles" \cite{scta}. In particular, any triangle $\triangle$ in Euclidean plane yields a triple of kissing circles with the centers at its vertices. This circles are sometimes called the {\it Soddy circles} of triangle $\triangle$. Their radii are readily expressed through the sides of triangle $\triangle$. So one can compute all invariants of the arising poristic Steiner chains in terms of the sides of $\triangle$. 
	
	\newpage
	\section{Pedoe relations in terms of bending invariants}\label{sec3}

Let $n\geq 3$ be a natural number and $(R, r, d)$ with $R>r$ be a triple of positive numbers satisfying the relation 
	
\begin{equation} \label{pedoe-n}
	d^2 =(R-r)^2-4qRr,
\end{equation}
where $q=\tan^2(\frac{\pi}{n})$. The inequality $(R-r)^2-4qRr>0$ is called $n$-th Pedoe condition. It is known that such a triple defines a poristic system of Steiner $n$-chains with the parent circles having radii $R, r$ and the distance between their centers equal to $d$ (see, e.g., \cite{ped}). We will say that this is a pair of parent circles of order $n$ with the gauge $(R,r,d)$. The number $d$ will be called the {\it gap} of parent circles. Denote by $a= \frac{1}{r}, A=-\frac{1}{R}$ the so-called {\it bends} (signed curvatures) of the parent circles. Following \cite{scta}, for each Steiner $n$-chain $G$ and natural number $k$, we define the $k$-th {\it bending moment} $I_k^{(n)}(G)$ of chain $G$ as
\begin{equation} \label{bendingmoments1}
	I_k^{(n)}(G) = \sum_{j=1}^n b_j^k,
\end{equation}
where $b_j$ are the bends of the circles in $G.$ As was shown in \cite{scta}, the first $n-1$ bending moments $I_k^{(n)}(G)$ are invariant in a poristic family of poristic Steiner $n$-chains with the given parent circles, so we denote them simply by $I_k^{(n)} = I_k^{(n)}(R,r,d)$ and call $I_k^{(n)}$ the {\it invariant bending moments} of poristic Steiner chains. Clearly, they can be expressed through radii $R, r$ of the parent circles or, equivalently, through their bends (signed curvatures) $A, a$. For $n=4$, the explicit formulas for the invariant bending moments $I_j^{(4)}, j=1,2,3,$ have been obtained in \cite{bidi} using some general results given in \cite{yiu}. It should be noted that this task was not considered neither in \cite{scta}, nor in \cite{yiu}. \\

For further use we reproduce the explicit formulas for the bending invariants of Steiner $4$-chains given in \cite{bidi}.

	\begin{thm} \label{invariants}
	For a given triple $(R, r, d)$ satisfying the fourth Pedoe relation, the three bending invariants are given by 
	
	\begin{equation}
		I_1^{(4)} = \frac{2(R-r)}{Rr},
	\end{equation}
	
	\begin{equation}
		I_2^{(4)} = \frac{3R^2 - 10Rr + 3r^2}{2R^2r^2},
	\end{equation}
	
	\begin{equation}
		I_3^{(4)} = \frac{5R^3 - 27R^2r + 27Rr^2 - 5r^3}{4R^3r^3}.
	\end{equation}
	\end{thm}

 The same invariants can now be easily expressed in terms of the bends of parent circles $A=-\frac{1}{R},a=\frac{1}{r}$ as follows. \cite{bidi,dikh}

\begin{cor} \label{c1}
\begin{equation} \label{moment41}
	I_1^{(4)} = 2(A+a),
\end{equation}

\begin{equation} \label{moment42}
	I_2^{(4)} = \frac{3A^2 + 10Aa + 3a^2}{2},
\end{equation}

\begin{equation} \label{moment43}
	I_3^{(4)} = \frac{5A^3 + 27A^2a + 27Aa^2 + 5a^3}{4}.
\end{equation}
\end{cor}

As was observed in \cite{bidi}, the above formulas can be used to express the bends of parent circles of a Steiner $4$-chain through the first two bending invariants $I_1, I_2$. Below we give the corresponding explicit formulas and establish the exact areal of the pairs $I_1, I_2$ for Steiner $4$-chains, which is the first main result of this paper. \\

Namely, formally expressing the bends $A, a$ from the equations (\ref{moment41}),
(\ref{moment42}) we get the following expressions:

\begin{cor} \label{c2}
	The following formulas hold for the admissible vales of $I_1, I_2$ 
	
	\begin{equation}\label{a}
		a=\frac{I_{1}}{4}+\frac{1}{2}\sqrt{I_{1}^{2}-2I_{2}},
	\end{equation}
	\begin{equation}\label{A}
		A=\frac{I_{1}}{4}-\frac{1}{2}\sqrt{I_{1}^{2}-2I_{2}}.
	\end{equation}        
\end{cor}

Notice that the choice of the signs in the above formulas is dictated by the above definition of bends as signed curvatures. This done, we need to take into account the feasibility domain of these expressions and $4$-th Pedoe relation, which yields the following result. 

\begin{thm} \label{areal1}
	The areal of all feasible pairs $I_1, I_2$ for Steiner $4$-chains coincides with the following domain
       \begin{equation}\label{invariantsareal}
              \{I_1 > 0, I_2 > 0, \frac{8}{3} I_2 < I_{1}^{2} < 4I_{2}\}.
        \end{equation}         
\end{thm}

{\bf Proof}
Rewriting the fourth Pedoe relation in terms of $A,a$, we have following relation,

\begin{equation}\label{Pedoe}
	d^{2}=\frac{(A+a)^{2}+4Aa}{A^{2}a^{2}}.
\end{equation}

Then, if we insert equations (\ref{a}),(\ref{A}) in (\ref{Pedoe}) and simplify, we get 

\begin{equation}\label{d4}
	d^{2}=\frac{128\left(4I_{2}-I_{1}^{2}\right)}{9I_{1}^{4}-48I_{1}^{2}I_{2}+64I_{2}^{2}}.
\end{equation}

$d^{2}$ should be positive, so we can write the following inequality and solve,

$$\frac{128\left(4I_{2}-I_{1}^{2}\right)}{\left(3I_{1}^{2}-8I_{2}\right)^{2}}>0.$$ \
Hence
$$4I_{2}-I_{1}^{2}>0,$$ \
and finally we have
$$I_{1}^{2}<4I_{2}.$$ \\
Also we know that (\ref{A}) should be negative, so 

$$\frac{I_{1}}{4}-\frac{1}{2}\sqrt{I_{1}^{2}-2I_{2}}<0,$$

$$\frac{1}{2}\sqrt{I_{1}^{2}-2I_{2}}>\frac{1}{4}I_{1},$$
squaring both sides of the inequality, we get
$$\frac{1}{4}\left(I_{1}^{2}-2I_{2}\right)>\frac{1}{16}I_{1}^{2},$$
after some simplification, we get
$$\frac{3}{16}I_{1}^{2}-\frac{1}{2}I_{2}>0,$$
hence
$$3I_{1}^{2}-8I_{2}>0,$$
and finally, we get that,
$$I_{1}^{2}>\frac{8}{3}I_{2}$$ \\
The proof of Theorem \ref{areal1} is complete. \\

For convenience of referring in the sequel we present a relevant corollary of Theorem \ref{areal1}.

\begin{cor}\label{c3}
	For any pair $I_1, I_2$ such that  $\{I_1 > 0, I_2 > 0, \frac{8}{3} I_2 < I_{1}^{2} < 4I_{2}\}$, the pair $(A,a),$  where
	\begin{equation*}
	a=\frac{I_{1}}{4}+\frac{1}{2}\sqrt{I_{1}^{2}-2I_{2}},
\end{equation*}
\begin{equation*}
	A=\frac{I_{1}}{4}-\frac{1}{2}\sqrt{I_{1}^{2}-2I_{2}}.
\end{equation*}        
	satisfies the $4$-th Pedoe relation.
\end{cor}

As was mentioned in \cite{scta}, all bending invariants are symmetric polynomials of parent bends $A, a$, which implies the following conclusion underlying our approach.

\begin{cor}
All bending invariants are algebraically expressible through $I_1, I_2$.
\end{cor} 

More precisely, by a fundamental algebraic theorem on algebraic dependence of polynomials there exists an algebraic relation between $I_1, I_2$ and $I_3$. For $n=4$, such a relation is especially simple and follows from the formulas given in Theorem \ref{invariants} and corollary \ref{c2}. Namely, the third bending moment is expressed as a polynomial in the first two bending moments. The exact form of this relation is given below and will be used in the sequel. \cite{dikh}

\begin{thm} \label{thirdmomentrelation}
	With the above notations and assumptions one has:
	\begin{equation} \label{moment3byfirst2}
		I_{3}^{(4)}=\frac{3}{4}I_1I_2-\frac{1}{8}I_{1} ^{3}. 
	\end{equation}
\end{thm}
{\bf Proof}
If we insert (\ref{a}) and (\ref{A}) in (\ref{moment43}), we get this algebraic relation.

\newpage
\section{Analytic expressions for the radii of Steiner circles} \label{sec 3}

In this section we give an analytic solution to Steiner $4$-porism and some of its consequences, which will be used to give an algorithgmic solution to the feasibility problem studied by K.Kiradjiev in \cite{kir}. \\

Let $(R,r,d)$ be a gauge of Steined $4$-porism. Recall that the term {\it axis of porism} refers to the straight line through the centers of parent circles. Obviously, this line is well defined only in non-concentric case, i.e. if $d\neq 0$. In concentric case any line though the common center of parent circles can be regarded as an axis of porism. Let us introduce a coordinate system with the origin at the center of the outer parent circle and $Ox$ axis along the axis of porism. Such a coordinate system will be called the {\it canonical coordinate system} of the porism considered. The main aim of this section is to give an exact formula for the radius $r(t)$ of poristic circle with the center $z(t)$ having the given polar angle $t$ in the canonical coordinate system. \\  
	
We will also need the following general property of poristic Steiner chains which is geometrically obvious and was rigorously proved in \cite{bidi}. 
	
	\begin{cor} \label{poristicrange1}
		For a given pair of parent circles with the gauge $(R, r, d)$, the minimal and maximal possible values of poristic circles $r_i$ are
		$$r_* = \frac{R - d - r}{2}, r^* = \frac{R + d - r}{2},$$
		while the minimal and maximal values of poristic curvatures are
		$$b_* = \frac{2}{R-d-r}, b^* = \frac{2}{R+d-r},$$
		respectively. For any $r \in [r_*, r^*],$ the poristic family contains a circle of radius $r$. 
	\end{cor}	
	
\begin{thm} \label{poristicradius}
\cite{dikh} The radius $r(t)$ of a poristic circle with the center having a given polar angle $t$ in the canonical coordinate system is given by:	
\begin{equation} \label{keyformula1}	
	r(t) = \frac{R^2 - 2dR\cos t + d^2 - r^2}{2(R + r - d\cos t)}.
\end{equation}
\end{thm}

{\bf Proof.}
We apply the cosine rule to the triangle formed by the centers $\Omega$ and $\omega$ of parent circles and the center $z(t)$ of the poristic triangle considered. From the very definition of Steiner chain follows that the sides of this triangle are $R - r(t)$, $r + r(t)$ and $d$. The result follows by noticing that the side $\vert\omega z(t)\vert$ is opposite to the angle $t$ and resolving the cosine rule with respect to $r(t)$. \\
 
This result will be used in the sequel as well as the following corollary which follows by inverting equation (\ref{keyformula1}).

\begin{cor} \label{polarangle}
	For any $\rho \in [r_*, r^*],$ there exist exactly two poristic circles with the radii equal to $\rho$ and their polar angles $t_{\pm}(\rho)$ are given by  
	
	\begin{equation} \label{keyformula2}	
		t_{\pm}(\rho) = \pm \cos^{-1} \Big(\frac{R^2 - r^2 + d^2 - 2(R+r)\rho}{2d(R-r)}\Big).
	\end{equation}
\end{cor}

The above formulas yield the so-called analytic solution to Steiner porism. We will also need some of their corollaries including the following general property of Steiner chains established in \cite{yiu}.

\begin{cor} \label{c7}
	The neighboring radii of a Steiner circle, in $4$-chains, having radius $r_0$ can be computed as the roots of the following quadrqtic equation:

\begin{equation} \label{yiuquadratics}
	\alpha x^2 + \beta x + \gamma,
\end{equation}
where
\begin{equation} \label{yiucoefficients}
	\begin{split}
		\alpha = 4R^2r^2r_{0}^{2}, \\
		\beta = -4Rrr_{0}^{2}(R-r), \\
		\gamma = [2Rr - (R-r)r_{0}]^2 + 4Rrr_{0}^{2}. \\
	\end{split}
\end{equation}
\end{cor} 

The proof follows directly from our analytic solution to Steiner porism. Corollary \ref{c7} immediately implies the next one.

\begin{cor}\label{c8}
	The sum of the neighboring radii is equal to $-\frac{\beta}{\alpha}.$
\end{cor}

\begin{rem}	
Using Theorem \ref{poristicradius} and Corollary \ref{polarangle} one obtain several other results on Steiner chains. In particular, one can easily compute the (canonical) coordinates of the center $z(t)$ as follows. For simplicity, we only describe this procedure for $n=4$. As was already mentioned, the distances between $z(t)$ and the centers of the outer and inner parent circles are $R-r(t)$ and $r + r(t),$ respectively. Hence the abscissa and ordinate of $z(t)$ are equal to $(R - r(t)) \cos t$ and $(R - r(t)) \sin t$ respectively. As we know, the radii $r_{\pm}$ of two neighboring poristic circles can be found as the roots of quadratic polynomial (\ref{yiuquadratics}). Having the values of $r_{\pm}$, the polar angles of the centers $z_{\pm}$ of the two neighboring poristic circles are given by equation (\ref{keyformula2}). For $n=3, 4$, this yields explicit formulas for the centers and radii of all poristic circles as functions of $t,$ which in turn yields explicit formulas for invariants of Steiner chains introduced in \cite{lamawi}. For $n=3, 4$, the corresponding formulas were given in \cite{dikh}. For $n\geq 5,$ one can do the same inductively but the resulting formulas are more difficult to handle and will be discussed elsewhere.
\end{rem}  

A chain which is invariant with respect to the reflection in the axis of porism $L$ will be called a {\it symmetric chain}. Such chains play an important role in the sequel. \\

We notice that a Steiner chain with the center of one its circles belonging to the axis of porism $L$ is symmetric and will be called {\it axial}. Moreover, a chain such that a pair of its circles has the tangency point in $L$ is also symmetric and will be called a {\it lateral chain}. \\ 

\newpage
\section{Feasibility problem for Steiner $4$-chains}\label{sec5}
	
We conclude the paper by applying our results to the so-called feasibility problem for the radii of Steiner $4$-chains considered in \cite{kir} and \cite{bidi}. Recall that this problem is formulated as follows. Given an ordered quadruple of positive numbers $(r_1, r_2, r_3, r_4)$ one is asked to find out if there exists a Steiner $4$-chain with such radii of its circles in the given order. To the best of our knowledge, this natural problem was considered for the first time in \cite{kir}, where the author gave its solution in one concrete case but did not discuss the general case. Another approach to solving the same problem was outlined in \cite{bidi}. In this section we show that the approach described in \cite{bidi} can be simplified using Theorem \ref{thirdmomentrelation}. \\
	
As was mentioned in the Introduction, an analogous feasibility problem has been solved in \cite{mazy} in a more general context of the so-called {\it $n$-flowers} which are the building blocks of the so-called {\it circle packings} playing important role in complex analysis and computational geometry \cite{ste}. \\
	
Recall that the algorithm suggested in \cite{bidi} involved several steps. Notice that we can easily compute the virtual parent radii $(\tilde R, \tilde r)$. If $\tilde R^2 - 6 \tilde R \tilde r + \tilde r^2 > 0$ then the pair $(\tilde R, \tilde r)$ is a feasible candidate for the radii of parent circles having the distance between their centers equal to $\tilde d = \sqrt{\tilde R^2 - 6 \tilde R \tilde r + \tilde r^2 }.$ Next, one has to verify that the given values of radii $r_i$ belong to the segment $[\tilde r_*, \tilde r^*]$. The final step of the algorithm given in \cite{bidi} required computing the "virtual $b$-moments" and comparing them with the actual $b$-moments $b_i = 1/r_i$. \\
	
Next one had to compute the first three moments $(I_1, I_2, I_3)$ of the quadruple $b_i = 1/r_i$ called the {\it actual moments of bends}. Then assuming the existence of a sought Steiner $4$-chain one uses the equations (\ref{moment41}) and (\ref{moment42}) to find the "virtual radii" $(\tilde R, \tilde r)$ of the "virtual parent circles". To this end it is sufficient to solve the system
	$${2(a+A)= I_1, 3A^2 + 10Aa + 3a^2 = 2I_2},$$
	which reduces to the following quadratic equation
	\begin{equation} \label{parentradii}
		16a^2 - 8I_1a + (8I_2 - 3I_1^2) = 0.
	\end{equation}

In view of our results the first step can be substituted by verifying that the actual first bending moments belong to the areal described in Theorem \ref{areal1}, which is obviously simpler. In particular, if this relation is not satisfied then the given quadruple of positive numbers $r_i$ cannot be realized as the radii of any Steiner $4$-chain. Then we check that all given radii lie in the segment $[\tilde r_*, \tilde r^*]$ for the virtual gauge. If this is the case, then each  of the given radii is realized in one of the poristic circles. Finally we choose one of the given radii, say $r_1$, and compute the radii of its neighbours in the virtual porism defined by the actual bending invariants $(I_1, I_2)$. 
If they coincide with the neighbors of $r_1$ in the given quadruple then it is easy to show that the given quadruple is realized in the virtual porism defined above. \\

In other words, we obtain a rigorous feasibility criterion for the radii of Steiner $4$-chains, which was the ultimate goal of this note. \\

Obviously, our results simplify the algorithm given in \cite{bidi}. For illustrative purposes, we conclude this section by giving a few concrete examples of application of our criterion. \\

{\bf Example 1.}
If we are given a quadruple of radii from \cite{kir} $\rho=(\frac{\sqrt{2}}{17},\frac{\sqrt{2}}{9},\sqrt{2},\frac{\sqrt{2}}{9})$, then the bends are $\varkappa = (12.0208,6.3639,0.7071.6.3639).$ So $(I_1,I_2)=(25.456,226)$ and we see that (\ref{moment3byfirst2}) is satisfied. In order to check if these radii are unordered realizable, we have to use our algorithm step by step: \\

{\bf Step 1.}
Insert $(I_1,I_2)=(25.456,226)$ in (\ref{invariantsareal}) we get that
$$602.666<648.007<904$$
So the Pedoe condition is satisfied. \\

{\bf Step 2.}
Calculate using (\ref{moment41}) and (\ref{moment42}) the "virtual radii" $(\tilde R, \tilde r)$ of the "virtual parent circles", find maximal and minimal radii, then check that the given radii fit into this interval. Indeed, 
$$\tilde R=1.5719, \tilde r=0.0748$$ hence $$\tilde r_*=0.748, \tilde r^*=2.827$$
So  $\rho \in (\tilde r_*, \tilde r^*)$. \\

Now it is easy to verify the remaining conditions, which implies that this quadruple is feasible. In fact,
the exact realization of these quadruple is given in the page 273 of \cite{kir}. 

{\bf Example 2.}  If we are given a quadruple of radii $\rho = (1, 2, 3, 4)$, then the bends are $\varkappa = (1, 0.5, 0.33(3), 0.25).$ So $(I_1,I_2,I_3)=(2.083,1.4214,1.1766)$ and we see that (\ref{moment3byfirst2}) is not satisfied, to chek if these radii are inorderd realizable, we have to use mentioned algorithm step by step again: \\

{\bf Step 1.}
Insert $(I_1,I_2)=(2.083,1.4214)$ in (\ref{invariantsareal}) we get that
 $$3.7904<4.3388<5.6856$$
this is realizable. \\

{\bf Step 2.}
Calculate using (\ref{moment41}) and (\ref{moment42}) the "virtual radii" $(\tilde R, \tilde r)$ of the "virtual parent circles", find maximal and minimal radii, then check, are given radii in this interval or not.
$$\tilde R=10.99, \tilde r=0.88, \tilde d=7.97 $$ hence $$\tilde r_*=1.07, \tilde r^*=9.04$$
So  $\rho \in (\tilde r_*, \tilde r^*)$.\\

{\bf Step 3.}
If we insert $\rho=1$ and virtual parent circles' radii in (\ref{keyformula2}) we get
$$t =0.1304116639.$$
Then if we insert this $t$ in (\ref{keyformula1}), we get that,
$$r(t)=1.239036813.$$
So $r(t)\neq 1$ and these radii are not realizable. \\

Analogous applications to the feasibility problem can be derived from the main result of \cite{mazy}. It would be interesting to compare those necessary conditions with the ones given above but we do not dwell on these aspects here for the reason of space.

\newpage
\section{Concluding remarks} \label{sec 6}

First of all, it is natural to search for generalizations of our results to Steiner $n$-chains for arbitrary $n$. For small $n$, this seems feasible in the same way as above, using (\ref{yiuquadratics}) and the invariance of the first $n-1$ moments of curvatures \cite{scta}. For $n=6$, this is especially simple since one can find the invariant bending moments in a similar way by finding the bends in the axial symmetric $6$-chain, which is calculated also in \cite{bidi}. Specifically, using the symmetry of the axial chain we can find all the bends and summing their powers one can obtain explicit formulas for the first five invariant bending moments. The resulting formulas are rather lengthy and will be published elsewhere. \\

For arbitrary $n\geq 5$, our approach yields certain necessary conditions in terms of the first two bending invariants. To this end one can compute the first two invariants in terms of the bends of parent circle using an inversion which take the given parent circles into a apir of concentric ones. Explicit formulas for such an inversion in terms of the gauge $(R,r,d)$ are well known and can be found, in particular, in Section 3 of \cite{scta}. For a concentric pair of parent circles all bending moments are easily computable and one compute their counterparts in original porism using the explicit formulas mentioned above. Then one can express the bends $(A,a)$ of parent cirlcles through the first two invariants and find teh exact areal of the admissible pairs $(I_1, I_2)$ in the same way as in the proof of our Theorem 2. This obviously yields necessary conditions for the realizability of a gien sequence of positive numbers as the radii of circles in a certain Steiner chain.  \\

Next, one can use the explicit formulas for the invariant bending moments to compute the canonical coordinates of the centres of axial symmetric chain. As is easy to realize this reduces to computing the intersections of pairs of quadrics, which can be done explicitly. This in turn can be used to solve the feasibility problem for the invariants introduced in \cite{lamawi}. \\

Finally, it would be interesting to compare our results with the necessary conditions of solvability of DFP given in \cite{mazy}. For $n=4$, this might appear helpful for obtaining 
a criterion of solvability of DFP.

\end{document}